\documentclass[12pt]{article}
\usepackage{amscd,amsmath,amssymb,amsthm,enumerate,epsfig,graphicx,mathrsfs,tikz,wasysym}

\newtheorem{lemma}{Lemma}
\newtheorem{proposition}{Proposition}
\newtheorem{corollary}{Corollary}
\newtheorem{remark}{Remark}
\newtheorem{definition}{Definition}

\title{Classification of Compatible Parabolic Subalgebras for Complex Symmetric Pairs}
\author{He Haian\\Department of Mathematics\\The Hong Kong University of Science and Technology}
\date{}

\begin{document}

\maketitle

\textbf{Abstract:} In this paper, we shall recall and arrange the relationship between hyperbolic elements and parabolic subalgebras at first. And then, we shall classify all the compatible parabolic subalgebras containing a $\tau$-stable Borel subalgebra for complex symmetric pair $(\mathfrak{g},\mathfrak{g}^\tau)$, where $\tau$ is a involutive automorphism of $\mathfrak{g}$.

\section{Background}

The compatibility of a parabolic subalgebra plays a very important role in representation theory. Restriction to a reductive Lie subalgebra of a generalized Verma module attached to a compatible parabolic subalgebra decomposes in the Grothendieck group with each summand finite multiplicity (Theorem 3.10 [\textbf{Ko}]). For complex symmetric pairs, T.Matsuki gave a very strong equivalent condition when a Borel subalgebra is $\mathfrak{g}^\tau$-compatible (Proposition 11 [\textbf{M}]), and T.Kobayashi generalized it; that is, a parabolic subalgebra $\mathfrak{p}$ of $\mathfrak{g}$ is $\mathfrak{g}^\tau$-compatible if and only if $\tau\mathfrak{p}=\mathfrak{p}$ (Lemma 4.5(2) [\textbf{Ko}]).

There are many books and articles studying symmetric space such as [\textbf{B}], [\textbf{C1}], [\textbf{C2}], [\textbf{H}], [\textbf{L1}] and [\textbf{L2}]. It is known that each complex symmetric space is a product of symmetric spaces of the form $G\times G/G$ such that $G$ is a complex connected simple Lie group and is diagonally embedded into $G\times G$, with symmetric spaces of the form $G/H$ such that $G$ is a complex connected simple Lie group and $H$ is its reductive subgroup. We shall make use of it together with the concept of hyperbolic element to classify all the $\mathfrak{g}^\tau$-compatible parabolic subalgebras containing a $\tau$-stable Borel subalgebra for complex symmetric pairs of this two types.

\section{Hyperbolic Element}

Suppose that $\mathfrak{g}$ is a complex semi-simple Lie algebra.
\begin{definition}
A semi-simple element $H\in\mathfrak{g}$ is said to be hyperbolic if the eigenvalues of $\mathrm{ad}(H)$ are all real.
\end{definition}
For a hyperbolic element $H$, we define the subalgebras
\begin{center}
$\mathfrak{u}_+\equiv\mathfrak{u}_+(H)$, $\mathfrak{l}\equiv\mathfrak{l}(H)$, $\mathfrak{u}_-\equiv\mathfrak{u}_-(H)$
\end{center}
as the sum of the eigenspaces with positive, zero, and negative eigenvalues, respectively. Then
\begin{center}
$\mathfrak{p}(H):=\mathfrak{l}(H)+\mathfrak{u}_+(H)$
\end{center}
is a Levi decomposition of a parabolic subalgebra of $\mathfrak{g}$.

Fix a Cartan subalgebra $\mathfrak{h}$ and a simple system $\Delta$ of $\mathfrak{g}$ which gives a positive system $\Phi^+$. Moreover, for $\alpha\in\Delta$, denote $T_\alpha$ to be the element in $\mathfrak{h}$ such that $\alpha(T_\alpha)$=1 and $\beta(T_\alpha)=0$ for all $\beta\in\Delta-\{\alpha\}$. By definition, the hyperbolic elements in $\mathfrak{h}$ are exactly \[\{\displaystyle{\sum_{\alpha\in\Delta}}c_\alpha T_\alpha\mid c_\alpha\in\mathbb{R}\}.\]

Because a standard parabolic subalgebra is determined by a subset of the simple system and a Borel subalgebra which it contains (Lemma 3.8.1(ii) [\textbf{CM}] or Proposition 5.90 [\textbf{K}]), we denote $\mathfrak{p}_\Pi$ to be the standard parabolic subalgebra corresponding to the subset $\Pi\subseteq\Delta(\mathfrak{g})$; namely, \[\mathfrak{p}_\Pi=\mathfrak{h}_\mathfrak{g}+\displaystyle{\sum_{\alpha\in\Phi^+}}\mathfrak{g}_\alpha+ \displaystyle{\sum_{\alpha\in\mathbb{N}\Pi\cap\Phi^+}}\mathfrak{g}_{-\alpha}\] where $\mathfrak{g}_\alpha$ is the root space of the root $\alpha$ and $\mathbb{N}$ is the set of non-negative integers.

If $H$ is a hyperbolic element in $\mathfrak{h}$, it gives a unique standard parabolic subalgebra $\mathfrak{p}(H)$ of $\mathfrak{g}$. However, a standard parabolic subalgebra of $\mathfrak{g}$ may correspond to a subset of hyperbolic elements in $\mathfrak{h}$. More precisely, for $\Pi\subseteq\Delta$, let's denote $\mathcal{H}(\mathfrak{h},\Delta,\mathfrak{p}_\Pi)$ to be the set of all the hyperbolic elements in $\mathfrak{h}$ which define the standard parabolic subalgebra $\mathfrak{p}_\Pi$ with respect to the simple system $\Delta$.
\begin{proposition}
$\mathcal{H}(\mathfrak{h},\Delta,\mathfrak{p}_\Pi)=\{\displaystyle{\sum_{\alpha\in\Delta-\Pi}}k_\alpha T_\alpha\mid k_\alpha\in\mathbb{R}^+\}$.
\end{proposition}
\begin{proof}
It is already known that, $\displaystyle{\sum_{\alpha\in\Delta-\Pi}}k_\alpha T_\alpha$ with $k_\alpha\in\mathbb{R}^+$ is a hyperbolic element in $\mathfrak{h}$. Let $\mathfrak{p}_\Pi=\mathfrak{l}_\Pi+\mathfrak{u}_+$ be the standard Levi decomposition. If $\beta\in\Phi^+(\mathfrak{l}_\Pi)$, $\beta$ is a positive integer combination of elements in $\Pi$, so $\beta(\displaystyle{\sum_{\alpha\in\Delta-\Pi}}k_\alpha T_\alpha)=0$. Else, if $\beta\in\Phi^+-\Phi^+(\mathfrak{l}_\Pi)$, $\beta$ is a positive integer combination of elements in $\Delta$ with at least one element in $\Delta-\Pi$, and then $\beta(\displaystyle{\sum_{\alpha\in\Delta-\Pi}}k_\alpha T_\alpha)>0$ since all $k_\alpha\in\mathbb{R}^+$. Also, $\mathrm{ad}\displaystyle{\sum_{\alpha\in\Delta-\Pi}}k_\alpha T_\alpha(\mathfrak{h})=0$. Hence, $\mathfrak{l}_\Pi=\mathfrak{l}(\displaystyle{\sum_{\alpha\in\Delta-\Pi}}k_\alpha T_\alpha)$ and $\mathfrak{u}_+=\mathfrak{u}_+(\displaystyle{\sum_{\alpha\in\Delta-\Pi}}k_\alpha T_\alpha)$. Thus, $\mathfrak{p}_\Pi=\mathfrak{p}(\displaystyle{\sum_{\alpha\in\Delta-\Pi}}k_\alpha T_\alpha)$. On the other hand, if $\mathfrak{p}_\Pi=\mathfrak{p}(H)$ for some hyperbolic element $H$ in $\mathfrak{h}$, write $H=\displaystyle{\sum_{\alpha\in\Delta}}k_\alpha T_\alpha$ with all $k_\alpha\in\mathbb{R}$. Let $\mathfrak{p}=\mathfrak{p}(H)=\mathfrak{l}(H)+\mathfrak{u}_+(H)$ be the Levi decomposition with respect to $H$. Because of the uniqueness of nilpotent radical, $\mathfrak{u}_+=\mathfrak{u}_+(H)$. Moveover, if $\beta\in\Phi^+(\mathfrak{l}_\Pi)$, $\mathfrak{g}_\beta$ lies in some eigenspace of $H$, and this eigenvalue must be 0, or it contradicts to $\mathfrak{u}_+=\mathfrak{u}_+(H)$. Hence, $\mathfrak{l}_\Pi\subseteq\mathfrak{l}(H)$, and compute the dimension, it is actually $\mathfrak{l}_\Pi=\mathfrak{l}(H)$. If $\beta\in\Pi$, $\mathfrak{g}_\beta\subseteq\mathfrak{l}_\Pi$, so $\mathrm{ad}H(\mathfrak{g}_\beta)=0$, and hence $k_\beta=\beta(H)=0$. Else, if $\beta\in\Delta-\pi$, $\mathfrak{g}_\beta\subseteq\mathfrak{u}_+$, so $\mathfrak{g}_\beta$ lies in a positive eigenspace of $\mathrm{ad}H$, and hence $k_\beta=\beta(H)>0$. Therefore, $H=\displaystyle{\sum_{\alpha\in\Delta-\Pi}}k_\alpha T_\alpha$ with all $k_\alpha\in\mathbb{R}^+$. This completes the proof.
\end{proof}
For a parabolic subalgebra $\mathfrak{p}$, we denote $\mathcal{H}(\mathfrak{p})$ to be the set of all the hyperbolic elements in $\mathfrak{g}$ which define $\mathfrak{p}$; namely, \[\mathcal{H}(\mathfrak{p})=\{H\in\mathfrak{g}\mid H\mathrm{hyperbolic},\mathfrak{p}=\mathfrak{p}(H)\}.\] By definition, $\mathcal{H}(\mathfrak{h},\Delta,\mathfrak{p}_\Pi)\subseteq\mathcal{H}(\mathfrak{p}_\Pi)$ and $\mathcal{H}(\mathfrak{p})\subseteq\mathfrak{p}$.

Write $G=\mathrm{Int}(\mathfrak{g})$, $P$ the parabolic subgroup of $G$ with Lie algebra $\mathfrak{p}$. Moreover, let \[P=LU_+=U_+L\] be the Levi decomposition in the group level corresponding to the Levi decomposition $\mathfrak{p}=\mathfrak{l}+\mathfrak{u}_+$ with $L$ reductive subgroup and $U_+$ unipotent radical.
\begin{lemma}
If $H\in\mathcal{H}(\mathfrak{p})$, then $\mathrm{Ad}(g)H\in\mathcal{H}(\mathrm{Ad}(g)\mathfrak{p})$ for $g\in G$.
\end{lemma}
\begin{proof}
If $H\in\mathcal{H}(\mathfrak{p})$, it gives a Levi decomposition $\mathfrak{p}=\mathfrak{l}(H)+\mathfrak{u}_+(H)$. For arbitrary $g\in G$, $\mathrm{Ad}(g)H$ is also a hyperbolic element in $\mathfrak{g}$ by definition, and one obtains another Levi decomposition $\mathrm{Ad}(g)\mathfrak{p}=\mathrm{Ad}(g)\mathfrak{l}(H)+\mathrm{Ad}(g)\mathfrak{u}_+(H)$. It is immediate to check by definition that $\mathrm{Ad}(g)\mathfrak{l}(H)=\mathfrak{l}(\mathrm{Ad}(g)H)$ and $\mathrm{Ad}(g)\mathfrak{u}_+(H)=\mathfrak{u}_+(\mathrm{Ad}(g)H)$. Therefore, $\mathrm{Ad}(g)\mathfrak{p}=\mathfrak{l}(\mathrm{Ad}(p)H)+\mathfrak{u}_+(\mathrm{Ad}(p)H)=\mathfrak{p}(\mathrm{Ad}(g)H)$, and then $\mathrm{Ad}(g)H\in\mathcal{H}(\mathrm{Ad}(g)\mathfrak{p})$.
\end{proof}
\begin{lemma}
If $H\in\mathcal{H}(\mathfrak{p})$, then $\mathrm{Ad}(p)H\in\mathcal{H}(\mathfrak{p})$ for $p\in P$.
\end{lemma}
\begin{proof}
Because $P$ normalizes $\mathfrak{p}$, $\mathrm{Ad}(p)\mathfrak{p}=\mathfrak{p}$ for $p\in P$. The conclusion holds immediately by Lemma 1.
\end{proof}
\begin{proposition}
$\mathcal{H}(\mathfrak{p}_\Pi)=\mathrm{Ad}(U_+)\mathcal{H}(\mathfrak{h},\Delta,\mathfrak{p}_\Pi):=\{\mathrm{Ad}(u)H\mid u\in U_+,H\in\mathcal{H}(\mathfrak{h},\Delta,\mathfrak{p}_\Pi)\}$.
\end{proposition}
\begin{proof}
$\mathcal{H}(\mathfrak{p}_\Pi)\supseteq\mathrm{Ad}(P_\Pi)\mathcal{H}(\mathfrak{h},\Delta,\mathfrak{p}_\Pi)$ is obvious by Lemma 2. On the other hand, if $H\in\mathcal{H}(\mathfrak{p}_\Pi)$, $H$ is a semi-simple element in $\mathfrak{p}_\Pi$. Hence, there exists an element $p\in P_\Pi$ such that $\mathrm{Ad}(p)H\in\mathfrak{h}$. Again by Lemma 2, $\mathrm{Ad}(p)H\in\mathcal{H}(\mathfrak{p}_\Pi)$. Therefore, $\mathrm{Ad}(p)H\in\mathcal{H}(\mathfrak{h},\Delta,\mathfrak{p}_\Pi)$, and $H\in\mathrm{Ad}(p^{-1})\mathcal{H}(\mathfrak{h},\Delta,\mathfrak{p}_\Pi)\subseteq\mathrm{Ad}(P_\Pi)\mathcal{H}(\mathfrak{h},\Delta,\mathfrak{p}_\Pi)$, so $\mathcal{H}(\mathfrak{p}_\Pi)\subseteq\mathrm{Ad}(P_\Pi)\mathcal{H}(\mathfrak{h},\Delta,\mathfrak{p}_\Pi)$. It is concluded that $\mathcal{H}(\mathfrak{p}_\Pi)=\mathrm{Ad}(P_\Pi)\mathcal{H}(\mathfrak{h},\Delta,\mathfrak{p}_\Pi)$. Because $[\mathfrak{l}_\Pi,\mathcal{H}(\mathfrak{h},\Delta,\mathfrak{p}_\Pi)]=0$, it follows that $\mathrm{Ad}(L_\Pi)(H)=H$ for all $H\in\mathcal{H}(\mathfrak{h},\Delta,\mathfrak{p}_\Pi)$. Therefore, $\mathcal{H}(\mathfrak{p}_\Pi)= \mathrm{Ad}(U_+L_\Pi)\mathcal{H}(\mathfrak{h},\Delta,\mathfrak{p}_\Pi)=\mathrm{Ad}(U_+)\mathcal{H}(\mathfrak{h},\Delta,\mathfrak{p}_\Pi)$.
\end{proof}
\begin{corollary}
$\mathcal{H}(\mathfrak{p}_\Pi)=\{\displaystyle{\sum_{\alpha\in\Delta-\Pi}}k_\alpha\mathrm{Ad}(u)T_\alpha\mid u\in U_+,k_\alpha\in\mathbb{R}^+\}$.
\end{corollary}
\begin{proof}
By Proposition 1 and Proposition 2, the conclusion is proved immediately.
\end{proof}
\begin{lemma}
$\mathcal{H}(\mathrm{Ad}(g)\mathfrak{p})=\mathrm{Ad}(g)\mathcal{H}(\mathfrak{p})$.
\end{lemma}
\begin{proof}
By Lemma 2, $\mathcal{H}(\mathrm{Ad}(g)\mathfrak{p})\supseteq\mathrm{Ad}(g)\mathcal{H}(\mathfrak{p})$, and $\mathcal{H}(\mathrm{Ad}(g)\mathfrak{p})\subseteq\mathrm{Ad}(g)\mathcal{H}(\mathfrak{p})$ symmetrically.
\end{proof}
Because every parabolic subalgebra of $\mathfrak{g}$ is conjugate by an element in $G$ to a standard parabolic subalgebra (Lemma 3.8.1(iii) [\textbf{CM}]), for each parabolic subalgebra $\mathfrak{p}$, there exists $\Pi\subseteq\Delta$ and $g\in G$ such that $\mathfrak{p}=\mathrm{Ad}(g)\mathfrak{p}_\Pi$.
\begin{proposition}
Suppose that $\mathfrak{p}$ is a parabolic subalgebra such that $\mathfrak{p}=\mathrm{Ad}(g)\mathfrak{p}_\Pi$ for some $\Pi\subseteq\Delta$ and $g\in G$. Then
\begin{center}
$\mathcal{H}(\mathfrak{p})=\{\displaystyle{\sum_{\alpha\in\Delta-\Pi}}k_\alpha\mathrm{Ad}(gu)T_\alpha\mid u\in U_+,k_\alpha\in\mathbb{R}^+\}$
\end{center}
where $U_+$ is the unipotent radical of $P_\Pi$.
\end{proposition}
\begin{proof}
By Lemma 3 and Corollary 1, the conclusion follows immediately.
\end{proof}

\section{Compatible Parabolic Subalgebra}

Let $\mathfrak{g}'$ be a reductive subalgebra of $\mathfrak{g}$, and $\mathfrak{p}$ a parabolic subalgebra of $\mathfrak{g}$.
\begin{definition}
We say $\mathfrak{p}$ is $\mathfrak{g}'$-$\emph{compatible}$ if there exists a hyperbolic element $H$ of $\mathfrak{g}'$ such that $\mathfrak{p}=\mathfrak{p}(H)$.
\end{definition}
If $\mathfrak{p}=\mathfrak{l}+\mathfrak{u}_+$ is $\mathfrak{g}'$-compatible, then $\mathfrak{p}':=\mathfrak{p}\cap\mathfrak{g}'$ becomes a parabolic subalgebra of $\mathfrak{g}'$ with Levi decomposition
\begin{center}
$\mathfrak{p}'=\mathfrak{l}'+\mathfrak{u}'_+:=(\mathfrak{l}\cap\mathfrak{g}')+(\mathfrak{u}_+\cap\mathfrak{g}')$.
\end{center}
According to the definition, $\mathfrak{p}$ is $\mathfrak{g}'$-compatible if and only if $\mathcal{H}(\mathfrak{p})\cap\mathfrak{g}'\neq\phi$.
\begin{lemma}
Let $\mathfrak{g}$ be a complex semi-simple Lie algebra with a reductive subalgebra $\mathfrak{g}'$. Suppose that $\mathfrak{p}$ is a parabolic subalgebra of $\mathfrak{g}$. Take an arbitrary maximal toral subalgebra $\mathfrak{h}'$ of $\mathfrak{p}':=\mathfrak{p}\cap\mathfrak{g}'$. Then $\mathfrak{p}$ is $\mathfrak{g}'$-compatible if and only if there exists a hyperbolic element $H$ in $\mathfrak{h}'$ such that $\mathfrak{p}=\mathfrak{p}(H)$.
\end{lemma}
\begin{proof}
The ``if'' part follows the definition immediately. Now if $\mathfrak{p}$ is a $\mathfrak{g}'$-compatible parabolic subalgebra, there is a hyperbolic element $H'\in\mathfrak{g}'$ such that $\mathfrak{p}=\mathfrak{p}(H')$. Moreover, $H'$ is a semi-simple element in $\mathfrak{p}$, so in $\mathfrak{p}'$, and hence there exists a maximal toral subalgebra $\mathfrak{t}'$ of $\mathfrak{p}'$ such that $H'\in\mathfrak{t}'$. Then $\mathfrak{h}'$ is conjugate to $\mathfrak{t}'$ by an element $x\in\mathrm{Int}(\mathfrak{p}')$, so $\mathrm{Ad}(x)H'\in\mathfrak{h}'$ which is also a hyperbolic element. Therefore, $\mathfrak{p}=\mathrm{Ad}(x)\mathfrak{p}=\mathfrak{p}(\mathrm{Ad}(x)H')$ by Lemma 3. Let $H=\mathrm{Ad}(x)H'$, and then the ``only if'' part is proved.
\end{proof}
\begin{remark}
Remember that ``reductive subalgebra'' means that each semi-simple element in $\mathfrak{g}'$ is also semi-simple in $\mathfrak{g}$; in other words, each Cartan subalgebra $\mathfrak{h}'$ of $\mathfrak{g}'$ is contained in some Cartan subalgebra $\mathfrak{h}$ of $\mathfrak{g}$. Thus $\mathfrak{h}\cap\mathfrak{g}'=\mathfrak{h}'$. If $\mathfrak{p}_\Pi$ is a parabolic subalgebra of $\mathfrak{g}$ containing $\mathfrak{h}$, then Lemma 4 tells us that $\mathfrak{p}_\Pi$ is $\mathfrak{g}'$-compatible if and only if $\mathcal{H}(\mathfrak{h},\Delta,\mathfrak{p}_\Pi)\cap\mathfrak{g}'\neq\phi$.
\end{remark}

\section{Classification for Complex Symmetric Pairs}

Let $\tau$ be an involutive automorphism of the complex semi-simple Lie algebra $\mathfrak{g}$ and $\mathfrak{g}^\tau$ be the set of fix points of $\tau$. Then $(\mathfrak{g},\mathfrak{g}^\tau)$ forms a complex symmetric pair. Let $\mathfrak{b}$ be a $\tau$-stable Borel subalgebra of $\mathfrak{g}$. We shall determine the $\mathfrak{g}'$-compatibility for all the parabolic subalgebras which contain $\mathfrak{b}$.
\begin{lemma}(Lemma 4.5 [$\mathbf{Ko}$])
A parabolic subalgebra is $\tau$-stable if and only if it is $\mathfrak{g}^\tau$-compatible.
\end{lemma}
\begin{proposition}
Let $(\mathfrak{g},\mathfrak{g}^\tau)$ be a complex symmetric pair with involution $\tau$, and let $\mathfrak{p}$ be a parabolic subalgebra of $\mathfrak{g}$ containing a $\tau$-stable Borel subalgebra $\mathfrak{b}$. Suppose that there is another $\tau$-stable Borel subalgebra $\mathfrak{b}'=\mathrm{Ad}(g)\mathfrak{b}$ for $g\in G=\mathrm{Int}(\mathfrak{g})$. If $\mathfrak{p}$ is $\mathfrak{g}^\tau$-compatible, so is $\mathrm{Ad}(g)\mathfrak{p}$.
\end{proposition}
\begin{proof}
Without confusion, let's denote $\tau$ to be the global involution of $G$ as well. Because both $\mathfrak{b}$ and $\mathrm{Ad}(g)\mathfrak{b}$ is $\tau$-stable, $\mathrm{Ad}(g^{-1}\tau(g))\mathfrak{b}=\mathrm{Ad}(g^{-1})\mathrm{Ad}(\tau(g))\tau\mathfrak{b}=\mathrm{Ad}(g^{-1})\tau(\mathrm{Ad}(g)\mathfrak{b})= \mathrm{Ad}(g^{-1})\mathrm{Ad}(g)\mathfrak{b}=\mathfrak{b}$. Hence, $g^{-1}\tau(g)\in N_G(\mathfrak{b})=N_G(B)=B$ where $B$ is the Borel subgroup of $G$ with Lie algebra $\mathfrak{b}$ by Proposition 6.4.9 [\textbf{Sp}]. Since $\mathfrak{p}$ is $\mathfrak{g}^\tau$-compatible, by Lemma 5, $\tau(\mathfrak{p})=\mathfrak{p}$, and it follows that $\tau(\mathrm{Ad}(g)\mathfrak{p})=\mathrm{Ad}(\tau(g))\tau(\mathfrak{p})=\mathrm{Ad}(g)\mathrm{Ad}(g^{-1}\tau(g))\mathfrak{p}=\mathrm{Ad}(g)\mathfrak{p}$ because $\mathfrak{p}$ contains $\mathfrak{b}$ and $g^{-1}\tau(g)\in B$ normalizes $\mathfrak{p}$. Finally by Lemma 5 again, $\mathrm{Ad}(g)\mathfrak{p}$ is $\mathfrak{g}^\tau$-compatible.
\end{proof}
\begin{remark}
By Lemma 5, choosing a $\tau$-stable Borel subalgebra is equivalent to choosing a $\mathfrak{g}^\tau$-compatible Borel subalgebra. Proposition 4 indicates that we are able to choose an arbitrary $\tau$-stable Borel subalgebra.
\end{remark}
As mentioned in Section 1, two types of complex symmetric pairs are worth studying:
\begin{enumerate}[type 1]
\item $(\mathfrak{g}\oplus\mathfrak{g},\mathfrak{g})$, where $\mathfrak{g}$ is a complex simple Lie algebra and is diagonally embedded into $\mathfrak{g}\oplus\mathfrak{g}$;
\item $(\mathfrak{g},\mathfrak{g}^\tau)$, where $\mathfrak{g}$ is a complex simple Lie algebra with involution $\tau$.
\end{enumerate}
As to type 1, the following theorem solves the problem completely.
\begin{proposition}
Let $(\mathfrak{g}\oplus\mathfrak{g},\mathfrak{g})$ be a complex symmetric pair, where $\mathfrak{g}$ is a complex semi-simple Lie algebra and is diagonally embedded into $\mathfrak{g}\oplus\mathfrak{g}$. Then the $\mathfrak{g}$-compatible parabolic subalgebras of $\mathfrak{g}\oplus\mathfrak{g}$ are exactly $\mathfrak{p}\oplus\mathfrak{p}$ for $\mathfrak{p}$ parabolic subalgebra of $\mathfrak{g}$.
\end{proposition}
\begin{proof}
Let $\tau$ denote the corresponding involution, and then $\tau:\mathfrak{g}\oplus\mathfrak{g}\rightarrow\mathfrak{g}\oplus\mathfrak{g}$ is given by $(X,Y)\mapsto(Y,X)$. It is known that the parabolic subalgebras of $\mathfrak{g}\oplus\mathfrak{g}$ are exactly $\mathfrak{p}_1\oplus\mathfrak{p}_2$ for $\mathfrak{p}_1$ and $\mathfrak{p}_2$ parabolic subalgebras of $\mathfrak{g}$. By Lemma 5, $\mathfrak{p}_1\oplus\mathfrak{p}_2$ is $\mathfrak{g}$-compatible if and only if $\tau(\mathfrak{p}_1\oplus\mathfrak{p}_2)=\mathfrak{p}_1\oplus\mathfrak{p}_2$, which is equivalent to $\mathfrak{p}_1=\mathfrak{p}_2$ according to the definition of $\tau$.
\end{proof}
From now on, only complex symmetric pairs of type 2 will be focused on. The following two lemmas are almost obvious.
\begin{lemma}
Let $\mathfrak{g}$ be a complex semi-simple Lie algebra and $\mathfrak{g}'$ be its reductive subalgebra. Then $\mathfrak{g}$ is $\mathfrak{g}'$-compatible.
\end{lemma}
\begin{proof}
Notice that $H=0\in\mathfrak{g}'$ is a hyperbolic element defining $\mathfrak{g}$, and the conclusion follows.
\end{proof}
\begin{lemma}
Let $\mathfrak{g}$ be a complex semi-simple Lie algebra and $\mathfrak{g}'$ be its reductive subalgebra. If $\mathrm{Rank}_\mathbb{C}\mathfrak{g}=\mathrm{Rank}_\mathbb{C}\mathfrak{g}'$, then there is a Cartan subalgebra $\mathfrak{h}$ of $\mathfrak{g}$ such that each parabolic subalgebra $\mathfrak{p}$ of $\mathfrak{g}$ containing $\mathfrak{h}$ is $\mathfrak{g}'$-compatible.
\end{lemma}
\begin{proof}
Choose a Cartan subalgebra $\mathfrak{h}'$ of $\mathfrak{g}'$, and then it is contained a Cartan subalgebra $\mathfrak{h}$ of $\mathfrak{g}$. Because $\mathrm{Rank}_\mathbb{C}\mathfrak{g}=\mathrm{Rank}_\mathbb{C}\mathfrak{g}'$, $\mathfrak{h}'=\mathfrak{h}$. It follows that each parabolic parabolic subalgebra $\mathfrak{p}$ of $\mathfrak{g}$ containing $\mathfrak{h}$ has a hyperbolic element in $\mathfrak{h}'$ which defines it. Therefore, $\mathfrak{p}$ is $\mathfrak{g}'$-compatible.
\end{proof}
Table 1 lists all the complex symmetric pairs $(\mathfrak{g},\mathfrak{g}^\tau)$ with $\mathfrak{g}$ simple such that $\mathrm{Rank}_\mathbb{C}\mathfrak{g}=\mathrm{Rank}_\mathbb{C}\mathfrak{g}'$.
\begin{proposition}
For each complex symmetric pair listed in Table 1, there exists a Cartan subalgebra $\mathfrak{h}$ of $\mathfrak{g}$ such that all the parabolic subalgebras $\mathfrak{p}$ of $\mathfrak{g}$ containing $\mathfrak{h}$ are $\mathfrak{g}^\tau$-compatible.
\end{proposition}
\begin{proof}
Because each complex symmetric pair listed in Table 1 has equal complex ranks, the conclusion follows Lemma 7.
\end{proof}
\begin{table}
\begin{center}
\caption{complex symmetric pairs $(\mathfrak{g},\mathfrak{g}^\tau)$ with $\mathfrak{g}$ simple and equal ranks}
\begin{tabular}{|c|c|c|}
\hline
$\mathfrak{g}$&$\mathfrak{g}^\tau$&$\mathrm{Rank}_\mathbb{C}\mathfrak{g}=\mathrm{Rank}_\mathbb{C}\mathfrak{g}^\tau$\\
\hline
$sl(m+n,\mathbb{C})$&$s(gl(m,\mathbb{C})\oplus gl(n,\mathbb{C}))$&$m+n$\\
\hline
$so(2m+n,\mathbb{C})$&$so(2m,\mathbb{C})\oplus so(n,\mathbb{C})$&$m+[\frac{n}{2}]$\\
\hline
$so(2n,\mathbb{C})$&$gl(n,\mathbb{C})$&$n$\\
\hline
$sp(2m+2n,\mathbb{C})$&$sp(2m,\mathbb{C})\oplus sp(2n,\mathbb{C})$&$m+n$\\
\hline
$sp(2n,\mathbb{C})$&$gl(n,\mathbb{C})$&$n$\\
\hline
$\mathfrak{e}_6$&$\mathfrak{e}_6$&$6$\\
\hline
$\mathfrak{e}_6$&$sl(6,\mathbb{C})\oplus sl(2,\mathbb{C})$&$6$\\
\hline
$\mathfrak{e}_6$&$so(10,\mathbb{C})\oplus so(2,\mathbb{C})$&$6$\\
\hline
$\mathfrak{e}_7$&$\mathfrak{e}_7$&$7$\\
\hline
$\mathfrak{e}_7$&$sl(8,\mathbb{C})$&$7$\\
\hline
$\mathfrak{e}_7$&$so(12,\mathbb{C})\oplus sp(2,\mathbb{C})$&$7$\\
\hline
$\mathfrak{e}_7$&$so(2,\mathbb{C})\oplus\mathfrak{e}_6$&$7$\\
\hline
$\mathfrak{e}_8$&$\mathfrak{e}_8$&$8$\\
\hline
$\mathfrak{e}_8$&$so(16,\mathbb{C})$&$8$\\
\hline
$\mathfrak{e}_8$&$sp(2,\mathbb{C})\oplus\mathfrak{e}_7$&$8$\\
\hline
$\mathfrak{f}_4$&$\mathfrak{f}_4$&$4$\\
\hline
$\mathfrak{f}_4$&$so(9,\mathbb{C})$&$4$\\
\hline
$\mathfrak{f}_4$&$sp(6,\mathbb{C})\oplus sp(2,\mathbb{C})$&$4$\\
\hline
$\mathfrak{g}_2$&$\mathfrak{g}_2$&$2$\\
\hline
$\mathfrak{g}_2$&$so(4,\mathbb{C})$&$2$\\
\hline
\end{tabular}
\end{center}
\end{table}
What remain to study are $(sl(2n+1,\mathbb{C}),so(2n+1,\mathbb{C}))$, $(sl(2n,\mathbb{C}),so(2n,\mathbb{C}))$, $(sl(2n,\mathbb{C}),sp(2n,\mathbb{C}))$, $(so(2m+2n,\mathbb{C}),so(2m-1,\mathbb{C})\oplus so(2n+1,\mathbb{C}))$, $(\mathfrak{e}_6,sp(8,\mathbb{C}))$ and $(\mathfrak{e}_6,\mathfrak{f}_4)$, which are those complex symmetric pairs $(\mathfrak{g},\mathfrak{g}^\tau)$ with $\mathfrak{g}$ complex simple and $\mathrm{Rank}_\mathbb{C}\mathfrak{g}\neq\mathrm{Rank}_\mathbb{C}\mathfrak{g}'$.

\textbf{(i) $(sl(2n+1,\mathbb{C}),so(2n+1,\mathbb{C}))$}

Recall that $sl(2n+1,\mathbb{C})$ consists of $(2n+1)\times(2n+1)$ matrices with traces 0. Take all the diagonal matrices to be its Cartan subalgebra, i.e., $\mathfrak{h}_{sl(2n+1,\mathbb{C})}=\{\mathrm{diag}(a_1,a_2,\cdots,a_{2n+1})\mid\displaystyle{\sum_{i=1}^{2n+1}}a_i=0\}$, and take the simple system to be $\Delta(sl(2n+1,\mathbb{C}))=\{\alpha_k-\alpha_{k+1}\mid1\leq k\leq2n\}$ where $\alpha_k(\mathrm{diag}(a_1,a_2,\cdots,a_{2n+1}))=a_k$.

We may realize $so(2n+1,\mathbb{C})$ as the set of skew-adjoint matrices relative to the quadratic form $\displaystyle{\sum_{i=1}^{n+1}}z_iz_{2n+2-i}$, and then take all the diagonal matrices $\mathfrak{h}_{so(2n+1,\mathbb{C})}=\{\mathrm{diag}(b_1,b_2,\cdots,b_n,0,-b_n,-b_{n-1},\cdots,-b_1)\}$ to be the Cartan subalgebra of $so(2n+1,\mathbb{C})$.

It is obvious that $\mathfrak{h}_{so(2n+1,\mathbb{C})}\subseteq\mathfrak{h}_{sl(2n+1,\mathbb{C})}$. Moreover, let $H=\mathrm{diag}(n,n-1,\cdots,1,0,-1,-2,\cdots,-n)\in\mathfrak{h}_{so(2n+1,\mathbb{C})}$, and then $H$ is a hyperbolic element defining the standard Borel subalgebra of $sl(2n+1,\mathbb{C})$ with respect to $\Delta(sl(2n+1,\mathbb{C}))$; namely, the standard Borel subalgebra $\mathfrak{b}_{sl(2n+1,\mathbb{C})}$ is $so(2n+1,\mathbb{C})$-compatible.
\begin{proposition}
The standard $so(2n+1,\mathbb{C})$-compatible parabolic subalgebras of $sl(2n+1,\mathbb{C})$ are exactly those $\mathfrak{p}_\Pi$ such that either $\{\alpha_k-\alpha_{k+1},\alpha_{2n+1-k}-\alpha_{2n+2-k}\}\subseteq\Pi$ or $\{\alpha_k-\alpha_{k+1},\alpha_{2n+1-k}-\alpha_{2n+2-k}\}\cap\Pi=\phi$ for all $1\leq k\leq n$.
\end{proposition}
\begin{proof}
By Lemma 4, $\mathfrak{p}_\Pi$ is $so(2n+1,\mathbb{C})$-compatible if and only if $\mathcal{H}(\mathfrak{p}_\Pi)\cap\mathfrak{h}_{so(2n+1,\mathbb{C})}\neq\phi$. Because $\alpha_k-\alpha_{k+1}=\alpha_{2n+1-k}-\alpha_{2n+2-k}$ on $\mathfrak{h}_{so(2n+1,\mathbb{C})}$ for $1\leq k\leq n$, either both $\alpha_k-\alpha_{k+1}$ and $\alpha_{2n+1-k}-\alpha_{2n+2-k}$ lie in $\Pi$ or neither of them lies in $\Pi$. On the other hand, let $\Pi=\displaystyle{\bigcup_{j=1}^t}\{\alpha_{i_j}-\alpha_{i_j+1},\alpha_{2n+1-i_j}-\alpha_{2n+2-i_j}\}$ with $1\leq t\leq n$ and $1\leq i_1<i_2<\cdots<i_t\leq n$. Take $T_k=\mathrm{diag}(\underbrace{1,1,\cdots,1}_k,\underbrace{0,0,\cdots,0}_{2n-2k+1},\underbrace{-1,-1,\cdots,-1}_k)$ for $1\leq k\leq n$, then $(\alpha_k-\alpha_{k+1})(T_k)= (\alpha_{2n+1-k}-\alpha_{2n+2-k})(T_k)=1$ and $\alpha(T_k)=0$ for $\alpha\in\Delta(sl(2n+1,\mathbb{C}))-\{\alpha_k-\alpha_{k+1},\alpha_{2n+1-k}-\alpha_{2n+2-k}\}$. Now take $H=\displaystyle{\sum_{\mbox{\tiny$\begin{array}{c}1\leq k\leq n\\k\neq i_j\end{array}$}}}T_k$, by Proposition 1, it defines $\mathfrak{p}_\Pi$. This shows that $\mathfrak{p}_\Pi$ is $so(2n+1,\mathbb{C})$-compatible.
\end{proof}

\textbf{(ii) $(sl(2n,\mathbb{C}),so(2n,\mathbb{C}))$}

The construction for this pair is almost same as \textbf{(i)}.

Similarly, $sl(2n,\mathbb{C})$ consists of $2n\times2n$ matrices with traces 0. Take all the diagonal matrices to be its Cartan subalgebra, i.e., $\mathfrak{h}_{sl(2n,\mathbb{C})}=\{\mathrm{diag}(a_1,a_2,\cdots,a_{2n})\mid\displaystyle{\sum_{i=1}^{2n}}a_i=0\}$, and take the simple system to be $\Delta(sl(2n,\mathbb{C}))=\{\alpha_k-\alpha_{k+1}\mid1\leq k\leq2n-1\}$ where $\alpha_k(\mathrm{diag}(a_1,a_2,\cdots,a_{2n}))=a_k$.

We may realize $so(2n,\mathbb{C})$ as the set of skew-adjoint matrices relative to the quadratic form $\displaystyle{\sum_{i=1}^n}z_iz_{2n+1-i}$, and then take all the diagonal matrices $\mathfrak{h}_{so(2n,\mathbb{C})}=\{\mathrm{diag}(b_1,b_2,\cdots,b_n,-b_n,-b_{n-1},\cdots,-b_1)\}$ to be the Cartan subalgebra of $so(2n,\mathbb{C})$.

It is obvious that $\mathfrak{h}_{so(2n,\mathbb{C})}\subseteq\mathfrak{h}_{sl(2n,\mathbb{C})}$. Moreover, let $H=\mathrm{diag} (n,n-1,\cdots,1,-1,-2,\\ \cdots,-n)\in\mathfrak{h}_{so(2n,\mathbb{C})}$, and then $H$ is a hyperbolic element defining the standard Borel subalgebra of $sl(2n,\mathbb{C})$ with respect to $\Delta(sl(2n,\mathbb{C}))$; namely, the standard Borel subalgebra $\mathfrak{b}_{sl(2n+1,\mathbb{C})}$ is $so(2n,\mathbb{C})$-compatible.
\begin{proposition}
The standard $so(2n,\mathbb{C})$-compatible parabolic subalgebras of $sl(2n,\mathbb{C})$ are exactly those $\mathfrak{p}_\Pi$ such that either $\{\alpha_k-\alpha_{k+1},\alpha_{2n-k}-\alpha_{2n+1-k}\}\subseteq\Pi$ or $\{\alpha_k-\alpha_{k+1},\alpha_{2n-k}-\alpha_{2n+1-k}\}\cap\Pi=\phi$ for all $1\leq k\leq n-1$.
\end{proposition}
\begin{proof}
The proof is almost same as that of Proposition 7 by replacing $2n+1$ with $2n$.
\end{proof}

\textbf{(iii)} $(sl(2n,\mathbb{C}),sp(2n,\mathbb{C}))$

Let's take the same Cartan subalgebra and simple system of $sl(2n,\mathbb{C})$ as in \textbf{(ii)}.

In most of the time, $sp(2n,\mathbb{C})$ is realized as \[\{\left(\begin{array}{cc}A&B\\C&-A^t\end{array}\right)\mid u,v\in\mathbb{C}^3,A\in gl(3,\mathbb{C}),B,C \mathrm{symmetric}\}.\] However, if conjugated by an element in $sl(2n,\mathbb{C})$ given by the block diagonal matrix $\mathrm{diag}(I_n,J_n)$ or $\mathrm{diag}(I'_n,J_n)$ depending on their determinants, where $I_n$ is $n\times n$ identity matrix, $I'_n$ denotes $\mathrm{diag}(-1,I_{n-1})$ and $J_n$ denotes the $n\times n$ anti-diagonal matrix with all anti-diagonal entries 1, all the diagonal matrices are $\{\mathrm{diag}(b_1,b_2,\cdots,b_n,-b_n,-b_{n-1},\cdots,-b_1)\}$ which form a Cartan subalgebra of $sp(2n,\mathbb{C})$.

Now things are totally same as \textbf{(ii)}. That is, the standard Borel subalgebra of $sl(2n,\mathbb{C})$ are $sp(2n,\mathbb{C})$-compatible. Moreover, the following conclusion holds.
\begin{proposition}
The standard $sp(2n,\mathbb{C})$-compatible parabolic subalgebras of $sl(2n,\mathbb{C})$ are exactly those $\mathfrak{p}_\Pi$ such that either $\{\alpha_k-\alpha_{k+1},\alpha_{2n-k}-\alpha_{2n+1-k}\}\subseteq\Pi$ or $\{\alpha_k-\alpha_{k+1},\alpha_{2n-k}-\alpha_{2n+1-k}\}\cap\Pi=\phi$ for all $1\leq k\leq n-1$.
\end{proposition}
\begin{proof}
The proof is totally same as that of Proposition 8.
\end{proof}

\textbf{(iv)} $(so(2m+2n,\mathbb{C}),so(2m-1,\mathbb{C})\oplus so(2n+1,\mathbb{C}))$

As constructed in \textbf{(ii)}, $so(2m+2n,\mathbb{C})$ is the set of skew-adjoint matrices relative to the quadratic form $\displaystyle{\sum_{i=1}^{m+n}}z_iz_{2m+2n+1-i}$. Take all the diagonal matrices $\mathfrak{h}_{so(2m+2n,\mathbb{C})}=\{\mathrm{diag}(a_1,a_2,\cdots,a_{m+n},-a_{m+n},-a_{m+n},\cdots,-a_1)\}$ to be the Cartan subalgebra of $so(2m+2n,\mathbb{C})$, and take the simple system to be $\Delta(so(2m+2n,\mathbb{C}))=\{\alpha_k-\alpha_{k+1}\mid1\leq k\leq m+n-1\}\cup\{\alpha_{m+n-1}+\alpha_{m+n}\}$ where $\alpha_k(\mathrm{diag}(a_1,a_2,\cdots,a_{m+n},-a_{m+n},-a_{m+n},\cdots,-a_1))=a_k$.

We may realize $so(2m-1,\mathbb{C})\oplus so(2n+1,\mathbb{C})$ as the set of direct sum of two skew-adjoint matrices relative to the direct product of two quadratic forms $(z_{m+n}^2+\displaystyle{\sum_{i=1}^{m-1}}z_iz_{2m+2n+1-i},z_{m+n+1}^2+\displaystyle{\sum_{i=m}^{m+n-1}}z_iz_{2m+2n+1-i})$, and then take all the diagonal matrices $\mathfrak{h}_{so(2m-1,\mathbb{C})\oplus so(2n+1,\mathbb{C})}=\{\mathrm{diag}(b_1,b_2,\cdots,b_{m-1},0,-b_{m-1},-b_{m-2},\cdots,-b_1)\}\oplus \{\mathrm{diag}(b_m,b_{m+1},\cdots,b_{m+n-1},0,-b_{m+n-1},-b_{m+n-2},\cdots,-b_{-m})\}\cong \{\mathrm{diag}(b_1,b_2,\cdots,\\b_{m+n-1},0,0,-b_{m+n-1},-b_{m+n-2},\cdots,-b_1)\}$ to be the Cartan subalgebra of $so(2m-1,\mathbb{C})\oplus so(2n+1,\mathbb{C})$.

It is obvious that $\mathfrak{h}_{so(2m-1,\mathbb{C})\oplus so(2n+1,\mathbb{C})}\subseteq\mathfrak{h}_{so(2m+2n,\mathbb{C})}$. Moreover, let $H=\mathrm{diag} (m+n-1,m+n-2,\cdots,1,0,0,-1,-2,\cdots,-m-n+1)\in\mathfrak{h}_{so(2m-1,\mathbb{C})\oplus so(2n+1,\mathbb{C})}$, and then $H$ is a hyperbolic element defining the standard Borel subalgebra of $so(2m+2n,\mathbb{C})$ with respect to $\Delta(so(2m+2n,\mathbb{C}))$; namely, the standard Borel subalgebra $\mathfrak{b}_{so(2m-1,\mathbb{C})\oplus so(2n+1,\mathbb{C})}$ is $so(2m-1,\mathbb{C})\oplus so(2n+1,\mathbb{C})$-compatible.
\begin{proposition}
The standard $so(2m-1,\mathbb{C})\oplus so(2n+1,\mathbb{C})$-compatible parabolic subalgebras of $so(2m+2n,\mathbb{C})$ are exactly those $\mathfrak{p}_\Pi$ such that $\{\alpha_{m+n-1}\pm\alpha_{m+n}\}\subseteq\Pi$ or $\{\alpha_{m+n-1}\pm\alpha_{m+n}\}\cap\Pi=\phi$.
\end{proposition}
\begin{proof}
First, because $\alpha_{m+n-1}+\alpha_{m+n}=\alpha_{m+n-1}-\alpha_{m+n}$ on $\mathfrak{h}_{so(2m-1,\mathbb{C})\oplus so(2n+1,\mathbb{C})}$, if there appears exactly one of them in $\Pi$, $\mathfrak{p}_\Pi$ is not $so(2m-1,\mathbb{C})\oplus so(2n+1,\mathbb{C})$-compatible by Lemma 4. Second, It is not easy to obtain $T_{\alpha_k-\alpha_{k+1}}=\mathrm{diag}(\underbrace{1,1,\cdots,1}_k,\underbrace{0,0,\cdots,0}_{2m+2n-2k},\underbrace{-1,-1,\cdots,-1}_k)\in \mathfrak{h}_{so(2m-1,\mathbb{C})\oplus so(2n+1,\mathbb{C})}$ for $1\leq k\leq m+n-1$. Hence, if $\{\alpha_{m+n-1}\pm\alpha_{m+n}\}\subseteq\Pi$, then $H=\displaystyle{\sum_{\alpha\in\Delta(so(2m+2n,\mathbb{C}))-\Pi}T_\alpha}$ just defines $\mathfrak{p}_\Pi$ by Proposition 1, so $\mathfrak{p}_\Pi$ is $so(2m-1,\mathbb{C})\oplus so(2n+1,\mathbb{C})$-compatible. Third, if $\Pi\cap\{\alpha_{m+n-1}\pm\alpha_{m+n}\}=\phi$, then $H=\displaystyle{\sum_{\alpha\in\Delta(so(2m+2n,\mathbb{C}))-\Pi-\{\alpha_{m+n-1}\pm\alpha_{m+n}\}}T_\alpha}+ \mathrm{diag}(1,1,\cdots,1,0,0,-1,-1,\cdots,-1)$ just defines $\mathfrak{p}_\Pi$, so $\mathfrak{p}_\Pi$ is $so(2m-1,\mathbb{C})\oplus so(2n+1,\mathbb{C})$-compatible. This finishes the proof.
\end{proof}

\textbf{(v)} $(\mathfrak{e}_6,sp(8,\mathbb{C}))$

Let $\mathfrak{h}_{\mathfrak{e}_6}$ be a Cartan subalgebra of $\mathfrak{e}_6$, and its dual space $\mathfrak{h}_{\mathfrak{e}_6}^*$ is complex linearly spanned by $\{\alpha_k\mid1\leq k\leq6\}$ with the following Cartan matrix. \[\left(\begin{array}{cccccc}2&0&0&-1&0&0\\0&2&-1&0&0&0\\0&-1&2&-1&0&0\\-1&0&-1&2&-1&0\\0&0&0&-1&2&-1\\0&0&0&0&-1&2\end{array}\right)\]
A simple system of $sp(8,\mathbb{C})$ is give by
\begin{eqnarray*}
\begin{array}{rclcrcl}
\beta_1&=&\alpha_1+\frac{1}{2}\alpha_3+\alpha_4+\frac{1}{2}\alpha_5,\\
\beta_2&=&\frac{1}{2}\alpha_2+\frac{1}{2}\alpha_6,\\
\beta_3&=&\frac{1}{2}\alpha_3+\frac{1}{2}\alpha_5,\\
\beta_4&=&\alpha_4,
\end{array}
\end{eqnarray*}
Where $\beta_4$ is the long root.
A Cartan subalgebra $\mathfrak{h}_{sp(8,\mathbb{C})}$ of $sp(8,\mathbb{C})$ contained in $\mathfrak{h}_{\mathfrak{e}_6}^*$ is complex linearly spanned by the co-roots:
\begin{eqnarray*}
\begin{array}{rclcrcl}
H_{\beta_1}&=&2H_{\alpha_1}+H_{\alpha_3}+2H_{\alpha_4}+H_{\alpha_5},\\
H_{\beta_2}&=&H_{\alpha_2}+H_{\alpha_6},\\
H_{\beta_3}&=&H_{\alpha_3}+H_{\alpha_5},\\
H_{\beta_4}&=&H_{\alpha_4}.
\end{array}
\end{eqnarray*}
Take $H=7H_{\beta_1}+10H_{\beta_2}+12H_{\beta_3}+13H_{\beta_4}$, and one checks easily that the standard Borel subalgebra $\mathfrak{b}_{\mathfrak{e}_6}$ with respect to the simple system $\Delta(\mathfrak{e}_6)$ is $sp(8,\mathbb{C})$-compatible.
\begin{proposition}
The standard $sp(8,\mathbb{C})$-compatible parabolic subalgebras of $\mathfrak{e}_6$ are exactly those $\mathfrak{p}_\Pi$ such that
\begin{enumerate}[(1)]
\item either $\{\alpha_2,\alpha_6\}\subseteq\Pi$ or $\{\alpha_2,\alpha_6\}\cap\Pi=\phi$;
\item either $\{\alpha_3,\alpha_5\}\subseteq\Pi$ or $\{\alpha_3,\alpha_5\}\cap\Pi=\phi$.
\end{enumerate}
\end{proposition}
\begin{proof}
First, because $\alpha_2=\alpha_6$ on $\mathfrak{h}_{sp(8,\mathbb{C})}$, if there appears exactly one of them in $\Pi$, $\mathfrak{p}_\Pi$ is not $sp(8,\mathbb{C})$-compatible by Lemma 4. The same for $\{\alpha_3,\alpha_5\}$. Second, it is not easy to obtain $T_{\alpha_1}=H_{\beta_1}\in\mathfrak{h}_{sp(8,\mathbb{C})}$ and $T_{\alpha_4}=3H_{\beta_1}+4H_{\beta_2}+5H_{\beta_3}+6H_{\beta_4}\in\mathfrak{h}_{sp(8,\mathbb{C})}$. Moreover, $T_{\{\alpha_2,\alpha_6\}}=H_{\beta_1}+2H_{\beta_2}+2H_{\beta_3}+2H_{\beta_4}\in\mathfrak{h}_{sp(8,\mathbb{C})}$ is the element such that $\alpha_2(T_{\{\alpha_2,\alpha_6\}})=\alpha_6(T_{\{\alpha_2,\alpha_6\}})=1$ and $\alpha_i(T_{\{\alpha_2,\alpha_6\}})=0$ for $i=1,3,4,5$. Similarly, $T_{\{\alpha_3,\alpha_5\}}=2H_{\beta_1}+3H_{\beta_2}+4H_{\beta_3}+4H_{\beta_4}$ is the element such that $\alpha_3(T_{\{\alpha_3,\alpha_5\}})=\alpha_5(T_{\{\alpha_3,\alpha_5\}})=1$ and $\alpha_i(T_{\{\alpha_3,\alpha_5\}})=0$ for $i=1,2,4,6$. Finally, given a subset $\Pi$ of $\Delta(\mathfrak{e}_6)$ satisfying the assumptions in the proposition, one can take the summation of some elements in $\{T_{\alpha_1},T_{\alpha_4},T_{\{\alpha_2,\alpha_6\}},T_{\{\alpha_3,\alpha_5\}}\}$ depending on $\Pi$ just as the proof of Proposition 10. This finishes the proof.
\end{proof}

\textbf{(vi)} $(\mathfrak{e}_6,\mathfrak{f}_4)$

A simple system of $\mathfrak{f}_4$ is give by
\begin{eqnarray*}
\begin{array}{rclcrcl}
\beta_1&=&\frac{1}{2}\alpha_2+\frac{1}{2}\alpha_6,\\
\beta_2&=&\frac{1}{2}\alpha_3+\frac{1}{2}\alpha_5,\\
\beta_3&=&\alpha_4,\\
\beta_4&=&\alpha_1,
\end{array}
\end{eqnarray*}
Where $\beta_1$ and $\beta_2$ are short roots while $\beta_3$ and $\beta_4$ are long roots.
A Cartan subalgebra $\mathfrak{h}_{\mathfrak{f}_4}$ of $\mathfrak{f}_4$ contained in $\mathfrak{h}_{\mathfrak{e}_6}^*$ is complex linearly spanned by the co-roots:
\begin{eqnarray*}
\begin{array}{rclcrcl}
H_{\beta_1}&=&H_{\alpha_2}+H_{\alpha_6},\\
H_{\beta_2}&=&H_{\alpha_3}+H_{\alpha_5},\\
H_{\beta_3}&=&H_{\alpha_4},\\
H_{\beta_4}&=&H_{\alpha_1}.
\end{array}
\end{eqnarray*}
Take $H=8H_{\beta_1}+15H_{\beta_2}+21H_{\beta_3}+11H_{\beta}$, and one checks easily that the standard Borel subalgebra $\mathfrak{b}_{\mathfrak{e}_6}$ with respect to the simple system $\Delta(\mathfrak{e}_6)$ is $\mathfrak{f}_4$-compatible.
\begin{proposition}
The standard $\mathfrak{f}_4$-compatible parabolic subalgebras of $\mathfrak{e}_6$ are exactly those $\mathfrak{p}_\Pi$ such that
\begin{enumerate}[(1)]
\item either $\{\alpha_2,\alpha_6\}\subseteq\Pi$ or $\{\alpha_2,\alpha_6\}\cap\Pi=\phi$;
\item either $\{\alpha_3,\alpha_5\}\subseteq\Pi$ or $\{\alpha_3,\alpha_5\}\cap\Pi=\phi$.
\end{enumerate}
\end{proposition}
\begin{proof}
The proof is parallel to that of Proposition 11, with $T_{\alpha_1}=H_{\beta_1}+2H_{\beta_2}+3H_{\beta_3}+2H_{\beta_4}$, $T_{\{\alpha_2,\alpha_6\}}=2H_{\beta_1}+3H_{\beta_2}+4H_{\beta_3}+2H_{\beta_4}$, $T_{\{\alpha_3,\alpha_5\}}=3H_{\beta_1}+6H_{\beta_2}+8H_{\beta_3}+4H_{\beta_4}$ and $T_{\alpha_4}=2H_{\beta_1}+4H_{\beta_2}+6H_{\beta_3}+3H_{\beta_4}$.
\end{proof}
Now the classification for complex symmetric pairs $(\mathfrak{g},\mathfrak{g}^\tau)$ of $\mathfrak{g}^\tau$-compatible parabolic subalgebras containing a $\tau$-stable Borel subalgebra  of $\mathfrak{g}$ is completely finished. Besides the pairs of type 1 and the pairs of type 2 with equal ranks, the classification for the six remaining pairs is given in Table 2. Note that the simple system of $\mathfrak{e}_6$ is given by the Cartan matrix in (\textbf{v}).
\begin{table}
\caption{$\mathfrak{g}^\tau$-compatible parabolic subalgebras $\mathfrak{p}_\Pi$ for complex symmetric pairs $(\mathfrak{g},\mathfrak{g}^\tau)$ with $\mathfrak{g}$ simple and unequal ranks}
\begin{center}
\begin{tabular}{|c|c|c|c|}
\hline
$\mathfrak{g}$&$\mathfrak{g}^\tau$&$\Delta$&$\Pi'\subseteq\Delta$ such that\\&&&either $\Pi'\subseteq\Pi$ or $\Pi'\cap\Pi=\phi$\\
\hline
$sl(n,\mathbb{C})$&$so(n,\mathbb{C})$&$\{\alpha_i-\alpha_{i+1}\mid1\leq i\leq n-1\}$&$\{\alpha_k-\alpha_{k+1},\alpha_{n-k}-\alpha_{n+1-k}\}$\\&&&for $1\leq k\leq [\frac{n}{2}]$\\
\hline
$sl(2n,\mathbb{C})$&$sp(2n,\mathbb{C})$&$\{\alpha_i-\alpha_{i+1}\mid1\leq i\leq2n-1\}$&$\{\alpha_k-\alpha_{k+1},\alpha_{2n-k}-\alpha_{2n+1-k}\}$\\&&&for $1\leq k\leq n-1$\\
\hline
&$so(2m-1,\mathbb{C})$&$\{\alpha_i-\alpha_{i+1}\mid1\leq i\leq m+n$&\\$so(2m+2n,\mathbb{C})$&$\oplus$&$-1\}\cup\{\alpha_{m+n-1}+\alpha_{m+n}\}$&$\{\alpha_{m+n-1}\pm\alpha_{m+n}\}$\\&$so(2n+1,\mathbb{C})$&&\\
\hline
$\mathfrak{e}_6$&$sp(8,\mathbb{C})$&$\{\alpha_i\mid1\leq i\leq6\}$&$\{\alpha_2,\alpha_6\}$ and $\{\alpha_3,\alpha_5\}$\\
\hline
$\mathfrak{e}_6$&$\mathfrak{f}_4$&$\{\alpha_i\mid1\leq i\leq6\}$&$\{\alpha_2,\alpha_6\}$ and $\{\alpha_3,\alpha_5\}$\\
\hline
\end{tabular}
\end{center}
\end{table}

\end{document}